\documentclass{amsart}
\usepackage{amsmath,amssymb,graphicx,latexsym}
\usepackage[usenames,dvipsnames]{color}
\usepackage{stmaryrd}		
\usepackage{hyperref}

\graphicspath{{figures/}}

\newtheorem{theorem}{Theorem}

\newtheorem{lemma}[theorem]{Lemma}
\newtheorem{proposition}[theorem]{Proposition}
\newtheorem*{theorem*}{Theorem}

\theoremstyle{definition}

\newtheorem*{example*}{Example}

\newtheorem*{definition*}{Definition}
\newtheorem*{notation*}{Notation}

\DeclareMathOperator{\total}{total}

\newcommand{\N}{\mathbb{N}}

\newcommand{\colonequal}{\mathrel{\mathop:}=}
\newcommand{\mytexttilde}{\raise.17ex\hbox{$\scriptstyle\sim$}}

\newcommand{\seq}[1]{\cite[\href{http://oeis.org/#1}{#1}]{OEIS}}

\begin{document}

\title[Universal Matrices for Counting $C$-nomial Coefficients]{Universal Matrices for Counting \\ Fibonomial and $C$-nomial Coefficients\\ by their $p$-adic Valuations}

\author{Arav Chand}
\email{
aravchando8@gmail.com
}
\begin{abstract}
Rowland found a matrix product formula for generating functions counting binomial coefficients by their $p$-adic valuations. A natural generalization of binomial coefficients was introduced by Knuth and Wilf defined by a sequence~$C$. We obtain analogous matrix product formulas counting these $C$-nomial coefficients when $C$ is a strong divisibility sequence. Surprisingly, the matrices are universal in the sense that they are independent of $C$. We further extend this product to $C$-multinomial coefficients.
\end{abstract}

\maketitle

\section{Introduction}\label{section:introduction}

Pascal's triangle has been studied by mathematicians for hundreds of years, especially for its numerous number theoretic properties. Fine~\cite{Fine1947} studied the divisibility of the entries on the $n$th row of Pascal's triangle by a prime $p$. Specifically, let $F_p(n)$ be the number of integers $m$ in the range $0 \leq m \leq n$ such that $\binom{n}{m}$ is not divisible by $p$, and let $n = n_{\ell}\cdots n_1n_0$ be the standard base-$p$ representation of $n$. Fine  showed that $$F_p(n)=(n_0+1)(n_1+1)\cdots(n_\ell+1).$$

Many authors have since been interested in generalizing Fine's theorem in various ways, such as counting $p$-adic valuations greater than $0$. One way this has been studied is through the generating function
$$
T_p(n, x) = \sum_{m=0}^n x^{\nu_p \left( \binom{n}{m} \right)}.
$$
This polynomial encodes the $p$-adic valuations of binomial coefficients, where $\nu_p(n)$ denotes the exponent of the highest power of $p$ dividing $n$. The coefficient on $x^i$ in $T_p(n, x)$ counts the number of $ m $ for which $ \nu_p\!\left( \binom{n}{m} \right) = i $. 

Rowland~\cite{Rowland2018} proved a matrix product formula for $T_p(n,x)$. For each $d \in \{0, 1, \dots, p~-~1\}$, let 
\[
M_p(d) = \begin{bmatrix}
d + 1 & p - d - 1 \\
dx & (p - d)x
\end{bmatrix}.
\]
Rowland showed 
\[
T_p(n, x) = \begin{bmatrix} 1 & 0 \end{bmatrix} M_p(n_0) M_p(n_1) \cdots M_p(n_\ell) \begin{bmatrix} 1 \\ 0 \end{bmatrix}.
\]
In this paper, we generalize this matrix product formula to analogues of binomial coefficients defined with respect to strong divisibility sequences. 

A sequence $C = C_1, C_2, \dots$ of nonzero integers is a \emph{strong divisibility sequence} if it satisfies \[
\gcd(C_n, C_m) = C_{\gcd(n, m)}
\] for all positive $n$ and $m$.
For example, the Fibonacci sequence $$(F_n)_{n\geq1}\colon 1, 1, 2, 3, 5, 8, 11, \dots$$ is an example of a strong divisibility sequence. Furthermore, all Lucas sequences $(U_n)_{n\geq1}$ that have relatively coprime $U_2$ and $U_3$ are also strong divisibility sequences. 

\begin{notation*}
For a strong divisibility sequence $C$, define the \emph{$ C $-orial} by $0!_C=1$ and $n!_C = C_n C_{n-1} \cdots C_1$ for positive $n$, and define the \emph{$ C $-nomial coefficient} by
\[
\binom{n}{m}_C = \frac{n!_C}{m!_C (n - m)!_C}
\] for non-negative $n$ and $m$. Knuth and Wilf~\cite{KnuthWilf1989} proved that $\binom{n}{m}_C$ is an integer.
Furthermore, let $ T_{p,C}(n,x) $ denote the generating polynomial for the $ p $-adic valuations of $ C $-nomials. 
That is, 
\[
T_{p,C}(n,x) = \sum_{m=0}^n x^{\nu_p\!\left( \binom{n}{m}_C \right)}.
\]
The coefficient on $x^i$ in $T_{p,C}(n,x)$ counts the number of $m$ for which $\nu_p\!\left(\binom{n}{m}_C\right)=i$.
\end{notation*}
We introduce some more notation to state the main theorems. For a strong divisibility sequence $C$, we define the \emph{rank of apparition} $ \alpha(m) $ to be the index of the first term in $ C $ divisible by $ m $, if such a term exists. 
If $\alpha(p)$ is not defined, the generating polynomial is trivially $T_{p,C}(n,x)=n+1$.
When $\alpha(p^k)$ is defined for all $k \geq 1$, we will also be interested in the sequence of ratios $(a_k(p))_{k\geq1}$ defined by
\[
a_1(p) := \alpha(p), \qquad a_2(p) := \frac{\alpha(p^2)}{\alpha(p)}, \qquad a_3(p) := \frac{\alpha(p^3)}{\alpha(p^2)}, \qquad \dots.
\] 
To give some intuition regarding this sequence, a value of $a_2(p)=1$ means the appearance of the factor $p^2$ in the sequence $C$ is at the same index as the appearance of the factor $p$. A value of $a_3(p)=p$ means the first term that divides $p^3$ is at an index exactly $p$ times the index of the first term that divides $p^2$. We use this ratio sequence to define three classes of primes.

\begin{definition*}
Let $p$ be a prime and $C$ a strong divisibility sequence such that $\alpha(p^k)$ is defined for all $k \geq 1$. Following Knuth and Wilf~\cite{KnuthWilf1989}, the prime $p$ is \emph{ideal} for $C$ if there exists an integer $s(p) \geq 1$ such that the sequence of ratios $(a_k(p))_{k \geq 2}$ satisfies
\[
a_k(p) =
\begin{cases}
1 & \text{if } 2 \leq k \leq s(p) \\
p & \text{if } k > s(p).
\end{cases}
\]
We say $p$ is \textit{acceptable} if there exists an integer $s(p) \geq 1$ such that \begin{align*}
a_k(p) &\in \N \quad \text{if } 2 \leq k \leq s(p) \\
a_k(p) &= p \quad \text{if } k > s(p).
\end{align*}
A prime is \emph{unacceptable} if it is not acceptable.
Note that every ideal prime is acceptable.
We mention that $s(p)$ is in fact $\nu_p\!\left(C_{\alpha(p)}\right)$. 
\end{definition*}

\begin{example*}
Take the Lucas sequence defined by $U_1=1$, $U_2=5$, and $U_n=5U_{n-1}+2U_{n-2}$ for $n > 2$. We choose this sequence as it demonstrates examples of ideal primes well. Consider $p=7$. We have $\alpha(7)=8$ since the first term divisible by $7$ is $U_8=120785=7^2\cdot2465$. This first term is also divisible by the next power of $p$, $7^2=49$, making $\alpha(7^2)=8$. Consequently, the ratio between them makes $a_2(7)=1$. The next power of $p$, $7^3=343$, first divides the term at the index $p\alpha(p^2)=7\cdot8=56$. For larger values of $k$, it follows from an exercise of Knuth~\cite[Exercise 3.2.2-11]{TAOCP} that $7^k$ first divides the term at index $7\alpha(7^{k-1})$, making the sequence $$(\alpha_k(7))_{k\geq1}=8,1,7,7,7,\ldots.$$
Since the last index for which $a_k(7)\neq7$ is $k=2$, we have $s(7) = 2$. Furthermore, because the ratios only take on the value $1$ for $2 \leq k \leq 2 = s(7)$, then $p=7$ is an ideal prime for $U$.
\end{example*}

\begin{example*}
To illustrate the generalization that acceptable primes present, consider $p=2$ in the Fibonacci sequence. With a similar analysis as the previous example, the sequence of ratios is $$(a_k(2))_{k\geq1}=3,2,1,2,2,2,\ldots.$$
Since the last index for which  $a_k(2)\neq2$ is $k=3$, we have $s(2)=3$. However, the ratios take on positive integer values for $2 \leq k \leq 3=s(2)$ that are not all $1$ like for the ideal primes. Thus, $p=2$ is an acceptable prime for the Fibonacci sequence.
\end{example*}

For the Lucas sequences, we can fully characterize the primes. All odd primes are ideal ~\cite[Exercise 3.2.2-11]{TAOCP}. The sole even prime, $p=2$, can be shown to be ideal whenever $U_2$ is even, or when $U_2$ odd and $U_3 \equiv 0 \pmod 4$. Else, $p=2$ is only an acceptable prime for the sequence $U$. Notably, there are no unacceptable primes for any Lucas sequence.

The matrix product formulas for the generating functions counting $C$-nomials are substantially simpler for ideal primes than acceptable primes. We thus state these separately.
In the ideal prime case, this initial vector can be stated explicitly. 

\begin{theorem}\label{thm:ideal_cnomial}
Let $ C $ be a strong divisibility sequence, and let $ p $ be an ideal prime for $ C $. Then for any $ n \geq 0 $, with the standard base-$p$ representation of $n = n_{\ell}\cdots n_1n_0$, and for any $ 0 \leq r < \alpha(p) $, we have 
\[
T_{p,C}(\alpha(p) \cdot n + r, x)
= \begin{bmatrix}
    r+1&(\alpha(p) - r - 1)x^{s(p) - 1}
\end{bmatrix} M_p(n_0) M_p(n_1) \cdots M_p(n_\ell) \begin{bmatrix} 1 \\ 0 \end{bmatrix}.
\]
\end{theorem}
In particular, the matrix product on the right side does not depend on $C$ except for the leftmost vector.

\begin{example*}
    Consider the Lucas sequence used above, defined by $U_1=1$, $U_2=5$, and $U_n=5U_{n-1}+2U_{n-2}$ for $n > 2$, and consider $p=7$. We calculate 
    $T_{7, U}(12,x)$ 
    in two ways. By its definition, the polynomial is generating from counting 
    $\binom{12}{m}_U$ 
    by their $7$-adic valuations for 
    $0 \leq m \leq 12$. 
    These coefficients and their $7$-adic valuations are \begin{center} 
    \begin{tabular}{c|r|c}
    $m$ & $\binom{12}{m}_U$ & $\nu_7\!\left(\binom{12}{m}_U\right)$ \\
    \hline
    0  & 1                                          & 0 \\
    1  & 100611585                                  & 0 \\
    2  & 376848881400519                            & 0 \\
    3  & 48655726732749786925                       & 0 \\
    4  & 217739746036832703362415                   & 0 \\
    5  & 33760841110473476348689725                 & 2 \\
    6  & 181372757631248546893092255                & 2 \\
    7  & 33760841110473476348689725                 & 2 \\
    8  & 217739746036832703362415                   & 0 \\
    9  & 48655726732749786925                       & 0 \\
    10 & 376848881400519                            & 0 \\
    11 & 100611585                                  & 0 \\
    12 & 1                                          & 0 \\
    \end{tabular}
    \end{center}
so $T_{7,U}(12,x)=10+3x^2$. 

Our matrix product provides an alternative and more computationally efficient way to find this polynomial. Since $p = 7$ is ideal, we use Theorem~\ref{thm:ideal_cnomial}. The matrices for $p=7$ are $M_7(0)$, $M_7(1)$, $\dots$, $M_7(6)$ which compute to 
$$M_7(0)=\begin{bmatrix}1&6\\0&7x\end{bmatrix}, \quad
M_7(1)=\begin{bmatrix}2&5\\x&6x\end{bmatrix},\quad \dots , \quad 
M_7(6)=\begin{bmatrix}7&0\\6x&x\end{bmatrix}.
$$ 
Recall for this Lucas sequence, $s(7)=2$ and $\alpha(p)=\alpha(7)=8$. Therefore, writing 
$12$ in the form $12=\alpha(p)\cdot n+r$ shows $n=1$ and $r=4$. 
Since $n=1=1_7$, the partial product excluding the leftmost vector evaluates to $$M_7(1)\begin{bmatrix}1\\0\end{bmatrix}=\begin{bmatrix}2&5\\x&6x\end{bmatrix}\begin{bmatrix}1\\0\end{bmatrix}=\begin{bmatrix}2\\x\end{bmatrix}.$$ 
The leftmost vector evaluates as $$\begin{bmatrix}
    r+1&(\alpha(p) - r - 1)x^{s(p) - 1}
\end{bmatrix}=\begin{bmatrix}
    5&3x
\end{bmatrix}.$$
Multiplying this by the partial product gives $$\begin{bmatrix}
    5&3x
\end{bmatrix}\begin{bmatrix}2\\x\end{bmatrix}=10+3x^2,$$ which exactly matches the $T_{7,U}(12,x)$ we calculated before. 
\end{example*}

\begin{example*}
    Elliptic divisibility sequences are a class of integer sequences where each sequence satisfies a nonlinear recurrence relation arising from division polynomials on elliptic curves. We consider the elliptic divisibility sequence associated with $y^2 + y = x^3 - x$, which is also a strong divisibility sequence~\seq{A006769}.
    This sequence starts $(C_n)_{n\geq1}\colon 1, 1, -1, 1, 2, -1, -3, -5, 7, -4, -23, 29, 59, 129,\ldots$. 

    We calculate $T_{2,C}(12,x)$ using our matrix product formula. For the sequence $C$, one can show $p=2$ is ideal, thus we use Theorem~\ref{thm:ideal_cnomial}. The matrices for $p=2$ are $M_2(0)$ and $M_2(1)$ which compute to 
    $$M_2(0)=\begin{bmatrix}1&1\\0&2x\end{bmatrix}, \quad M_2(1)=\begin{bmatrix}2&0\\x&x\end{bmatrix}.$$ 
    For this sequence, $\alpha(2)=5$ since the first index $n$ such that $2 \mid C_n$ is $n=5$. Furthermore, we find $s(2)=1$. Therefore, writing $12$ in the form $12=\alpha(p)\cdot n+r$ shows $n=2$ and $r=2$. Since $n=2=10_2$, the partial product excluding the leftmost vector evaluates to $$M_2(0)M_2(1)\begin{bmatrix}1\\0\end{bmatrix}=\begin{bmatrix}1&1\\0&2x\end{bmatrix}\begin{bmatrix}2&0\\x&x\end{bmatrix}\begin{bmatrix}1\\0\end{bmatrix}=\begin{bmatrix}2+x\\2x^2\end{bmatrix}.$$ 
The leftmost vector evaluates as $$\begin{bmatrix}
    r+1&(\alpha(p) - r - 1)x^{s(p) - 1}
\end{bmatrix}=\begin{bmatrix}
    3&2
\end{bmatrix}.$$
Multiplying this by the partial product gives $$T_{2,C}(12,x)=\begin{bmatrix}
    3&2
\end{bmatrix}\begin{bmatrix}2+x\\2x^2\end{bmatrix}=6+3x+4x^2.$$ 
\end{example*}

Now we turn to acceptable primes. We introduce notation to state Theorem~\ref{thm:acceptable_cnomial}. For an acceptable prime, the initial vector is defined in terms of
$$f_{p,C,\lambda}(r,r_1,r_2)=\lambda\sum_{j=1}^{s(p)}\frac{\alpha(p^{s(p)})}{\alpha(p^{j})} + \sum_{j=1}^{s(p)}\left(\! \left\lfloor\frac{r}{\alpha(p^{j})} \right\rfloor - \left\lfloor\frac{r_1}{\alpha(p^{j})} \right\rfloor - \left\lfloor\frac{r_2}{\alpha(p^{j})} \right\rfloor \right).$$
This accounts for the difference in behavior between an acceptable prime and an ideal prime for the initial interval $2 \leq k \leq s(p)$, and the consequential effect to the $p$-adic valuation of the $C$-nomial. This will seem more natural as a special case of the inital vector in Theorem~\ref{thm:acceptable_cmultinomial}.
\begin{notation*} Define a vector $u_{p,C}(r)=\begin{bmatrix}
    u_{p,C,0}(r)&u_{p,C,1}(r)
\end{bmatrix}$ where 
$$u_{p,C,\lambda}(r)=\sum_{\substack{(r_1,r_2) \in \{0, 1, \ldots, \alpha(p^{s(p)})-1\}^2 \\ r_1+r_2 = r+\lambda\alpha(p^{s(p)})}}x^{f_{p,C,\lambda}(r, r_1,r_2)-\lambda} .$$
One can check that this $u_{p,C}(r)$ specializes to the vector in Theorem~\ref{thm:ideal_cnomial} when $p$ is ideal.
\end{notation*}

\begin{example*}
    Consider the Fibonacci sequence $F$ with $p=2$. We compute the vectors $u_{p,F}(r)$ for all $r$. For the Fibonacci sequence, $s(2)=3$ and $\alpha(p^{s(p)})=\alpha(2^3)=6$. We demonstrate the calculation for $r=1$ where
$$u_{2,F}(1)=\begin{bmatrix}
    \sum\limits_{r_1=0}^{1}x^{f_{2,F,0}(1,r_1,1-r_1)}
    &
    \sum\limits_{r_1=2}^{6}x^{f_{2,F,1}(r,r_1,7-r_1)-1}
\end{bmatrix}.$$
The term $\sum\limits_{j=1}^{s(p)}\frac{\alpha(p^{s(p)})}{\alpha(p^{j})}=\sum\limits_{j=1}^{3}\frac{\alpha(2^{3})}{\alpha\left(2^{j}\right)}$ in $f$ evaluates to $\frac{\alpha(8)}{\alpha(2)}+\frac{\alpha(8)}{\alpha(4)}+\frac{\alpha(8)}{\alpha(8)} = \frac{6}{3}+\frac{6}{6}+\frac{6}{6}=4$. The exponents in the first and second entries of the vector thus evaluate to 
\[
\begin{array}{c|c}
(r_1, r_2) & f_{2,F,0}(1, r_1, r_2) \\
\hline
(0,1) & 0 \\
(1,0) & 0 \\
\end{array}
\quad\quad
\begin{array}{c|c}
(r_1, r_2) & f_{2,F,1}(1, r_1, r_2) -1 \\
\hline
(2,5) & 2 \\
(3,4) & 1 \\
(4,3) & 1 \\
(5,2) & 2 \\
\end{array}
\]
making the vector $$u_{2,F}(1)=\begin{bmatrix}
    2&2x+2x^2
\end{bmatrix}.$$
The vector can be computed for other values of $r$ similarly, finding
\[
\begin{array}{c|c}
r & u_{2,F}(r) \\
\hline
0 & \begin{bmatrix}1&x+4x^2 \end{bmatrix} \\
1 & \begin{bmatrix}2&2x+2x^2 \end{bmatrix} \\
2 & \begin{bmatrix}3&3x\end{bmatrix} \\
3 & \begin{bmatrix}2+2x&2x^2 \end{bmatrix} \\
4 & \begin{bmatrix}4+x&x^2 \end{bmatrix} \\
5 & \begin{bmatrix}6&0 \end{bmatrix} \\
\end{array}.
\] 
\end{example*}
\begin{theorem}\label{thm:acceptable_cnomial}

Let $ C $ be a strong divisibility sequence, and let $ p $ be an acceptable prime for $ C $. Then for any $ n \geq 0 $, with the standard base-$p$ representation of $n = n_{\ell}\cdots n_1n_0$, and for any $ 0 \leq r < \alpha(p^{s(p)}) $, we have 
\[
T_{p,C}(\alpha(p^{s(p)}) \cdot n + r, x)
= u_{p,C}(r) M_p(n_0) M_p(n_1) \cdots M_p(n_\ell) \begin{bmatrix} 1 \\ 0 \end{bmatrix}.
\]
\end{theorem}
As in Theorem ~\ref{thm:ideal_cnomial}, other than the leftmost vector, the right hand side is independent of $C$.
\begin{example*}
Consider the Fibonacci sequence $F$ with $p=2$. We demonstrate calculating $T_{2,F}(13, x)$ using our matrix product formula, which provides a more computationally efficient way to find this polynomial. Since $p = 2$ is acceptable but not ideal, we use Theorem~\ref{thm:acceptable_cnomial}. The matrices for $p=2$ are $M_2(0)$ and $M_2(1)$, which compute to $$M_2(0)=\begin{bmatrix}1&1\\0&2x\end{bmatrix},\quad 
M_2(1)=\begin{bmatrix}2&0\\x&x\end{bmatrix}.
$$ Recall for the Fibonacci sequence, $s(2)=3$ and $\alpha(p^{s(p)})=\alpha(2^3)=6$. Therefore, writing $13$ in the form $13=\alpha(p^{s(p)})\cdot n+r$ shows $n=2$ and $r=1$. Since $n=2=10_2$, the partial product excluding the leftmost vector evaluates to $$\begin{bmatrix}1&1\\0&2x\end{bmatrix}\begin{bmatrix}2&0\\x&x\end{bmatrix}\begin{bmatrix}1\\0\end{bmatrix}=\begin{bmatrix}2+x\\2x^2\end{bmatrix}.$$ 
Above we found the leftmost vector as $$u_{2,F}(1)=\begin{bmatrix}
    2&2x+2x^2
\end{bmatrix}.$$
Multiplying this by the partial product gives $$T_{2,F}(13,x) =\begin{bmatrix}
    2&2x+2x^2
\end{bmatrix}\begin{bmatrix}2+x\\2x^2\end{bmatrix}=4x^4+4x^3+2x+4.$$ 

Notice how if we wanted to calculate $T_{2,F}(14,x)$ instead, using the definition of the generating polynomial would require us to recalculate $\binom{14}{m}_C$ for all $0\leq m\leq 14$, and their $2$-adic valuations. However, using our formula, we would not need to change the partial product of the matrices. We would only need to change the initial vector, using $u_{2,F}(2)$ instead of $u_{2,F}(1)$.
\end{example*}

We generalize our formulas to multinomial coefficient analogues. For a $k$-tuple $\mathbf{m}=(m_1,m_2,\ldots m_k)$ of non-negative integers, define $$\total \mathbf{m}\colonequal m_1+m_2+\cdots m_k.$$
The multinomial coefficient is $$\binom{n}{m_1,\, m_2,\, \dots,\, m_k} = \frac{n!}{m_1!m_2!\cdots m_k!},$$ where $\total\mathbf{m}=n$.
The generating function studied by Rowland ~\cite{Rowland2018} is
\[
T_{p,k}(n, x) = \sum_{\substack{\mathbf{m} \in \N^k \\ \total\mathbf{m} = n}} x^{\nu_p\!\left(\binom{n}{m_1,\, m_2,\, \dots,\, m_k}\right)}.
\]
where $k\geq2$ and $\N$ is the set of non-negative integers including 0.
Rowland~\cite{Rowland2018} defined $k \times k$ matrices $M_{p,k}(d)$, and proved they can be similarly used to construct a matrix product formula for $T_{p,k}(n, x)$.

\begin{notation*}
For a strong divisibility sequence $C$, define the \emph{$C$-multinomial} as 
$$\binom{n}{m_1,\, m_2,\, \dots,\, m_k}_C=\frac{n!_C}{m_1!_C m_2!_C \cdots  m_k!_C}.$$
and let $T_{p,k,C}(n,x)$ denote the generating polynomial for the $p$-adic valuations of $C$-multinomials 
\[
T_{p,k,C}(n, x) = \sum_{\substack{\mathbf{m} \in \N^k \\ \total\mathbf{m} = n}} x^{\nu_p\!\left(\binom{n}{m_1,\, m_2,\, \dots,\, m_k}_C\right)}.
\]
\end{notation*}

\begin{notation*} 
Define the vector $u_{p,k,C}(r)=\begin{bmatrix}
    u_{p,k,C,0}(r)&u_{p,k,C,1}(r)&\cdots&u_{p,k,C,k-1}(r)
\end{bmatrix}$ where 
$$u_{p,k,C,\lambda}(r)=x^{(s(p)-1)\lambda}\sum_{j=0}^{\lambda}(-1)^j\binom{k}{j}\binom{r+(\lambda-j)\alpha(p)+k-1}{k-1}.$$
\end{notation*}
\begin{theorem}\label{thm:ideal_cmultinomial}
Let $ C $ be a strong divisibility sequence, and let $ p $ be an ideal prime for $ C $. Let $k \geq 2$ and $e=\begin{bmatrix}
    1&0&0&\cdots&0
\end{bmatrix}\in\mathbb{Z}^k$. Then for any $ n \geq 0 $, with the standard base-$p$ representation of $n = n_{\ell}\cdots n_1n_0$, and for any $ 0 \leq r < \alpha(p) $, we have
\[
T_{p,k,C}(\alpha(p) \cdot n + r, x)
= u_{p,k,C}(r) M_{p,k}(n_0) M_{p,k}(n_1) \cdots M_{p,k}(n_\ell) e^\intercal.
\]
\end{theorem}

Similar to $C$-nomials, for an acceptable prime $ p $ and $k\geq2$, the initial vector is defined in terms of
\[
f_{p,k,C,\lambda}(r,\mathbf{r})=\lambda\sum_{j=1}^{s(p)}\frac{\alpha(p^{s(p)})}{\alpha(p^{j})} + \sum_{j=1}^{s(p)}\left( \left\lfloor\frac{r}{\alpha(p^{j})} \right\rfloor - \sum_{i=1}^k\left\lfloor\frac{r_i}{\alpha(p^{j})} \right\rfloor \right).
\] 

We now state the theorem for acceptable primes for counting $C$-multinomials.
\begin{notation*} Define a vector $u_{p,k,C}(r)=\begin{bmatrix}
    u_{p,k,C,0}(r)&u_{p,k,C,1}(r)&\cdots&u_{p,k,C,k-1}(r)
\end{bmatrix}$ where 
$$u_{p,k,C,\lambda}(r)=\sum_{\substack{\mathbf{r} \in \{0, 1, \ldots, \alpha(p^{s(p)})-1\}^k \\ \total\mathbf{r} = r+\lambda\alpha(p^{s(p)})}}x^{f_{p,k,C,\lambda}(r, \mathbf{r})-\lambda} .$$
\end{notation*}

\begin{theorem}\label{thm:acceptable_cmultinomial}
Let $ C $ be a strong divisibility sequence, and let $ p $ be an acceptable prime for $ C $. Let $k \geq 2$ and $e=\begin{bmatrix}
    1&0&0&\cdots&0
\end{bmatrix}\in\mathbb{Z}^k$. Then for any $ n \geq 0 $, with the standard base-$p$ representation of $n = n_{\ell}\cdots n_1n_0$, and for any $ 0 \leq r < \alpha(p^{s(p)}) $, we have
\[
T_{p,k,C}(\alpha(p^{s(p)}) \cdot n + r, x)
= u_{p,k,C}(r) M_{p,k}(n_0) M_{p,k}(n_1) \cdots M_{p,k}(n_\ell) e^\intercal
\]
\end{theorem}

A sequence $(s_n)_{n\geq0}$, with entries in some field, is \emph{$p$-regular} if the vector space generated by the set of subsequences $\{(s_{p^en+i})_{n\geq 0} \colon e \geq 0 \text{ and }0\leq i \leq p^e-1\}$ is finite dimensional. Allouche and Shallit~\cite{AlloucheShallit1992} first introduced regular sequences and showed they have many desirable properties, making them a natural class. One such property emerges out of this class's origin in theoretical computer science and specifically, finite state automata: any sequence whose term at index $n$ may be generated as a function of the digits in the base $k$ representation of $n$ is a $k$-regular sequence. Rowland showed that $(T_p(n,x))_{n\geq0}$ and $(T_{p,k}(n,x))_{n\geq0}$ are each a $p$-regular sequence of polynomials~\cite{Rowland2018}. Our analogues to the sequence $C$ show that certain subsequences are similarly $p$-regular. That is, $(T_{p,C}(\alpha(p)\cdot n+r))_{n\geq0}$ for any $0 \leq r < \alpha(p)$ is a $p$-regular sequence, and $(T_{p,k,C}(\alpha(p^{s(p)})\cdot n+r))_{n\geq0}$ for any $0 \leq r < \alpha(p^{s(p)})$ is a $p$-regular sequence.

We set up necessary identities for the subsequent proofs in Section~\ref{section:valuationidentites}. We consider ideal and acceptable primes in separate sections, proving Theorems \ref{thm:ideal_cnomial} and \ref{thm:ideal_cmultinomial} in Section~\ref{section:ideal} dealing with ideal primes. We prove Theorem~\ref{thm:acceptable_cmultinomial} in Section~\ref{section:acceptable} dealing with acceptable primes. Note we do not provide a separate proof for Theorem~\ref{thm:acceptable_cnomial} as it is a subcase of Theorem~\ref{thm:acceptable_cmultinomial} that is not substantially simpler to prove. 
Since ideal primes are also acceptable and $C$-nomials are $C$-multinomials with $k=2$, Theorem~\ref{thm:acceptable_cmultinomial} is our most broad statement for a matrix product formula that counts generalized binomial coefficients by their $p$-adic valuation.

\section{Valuation Identities}\label{section:valuationidentites}
In this section we obtain identities for the $p$-adic valuation of the $C$-orial, $C$-nomial, and $C$-multinomial coefficients in the cases of ideal and acceptable primes.

We begin by finding a Legendre-like formula for the $C$-orial:
\begin{proposition}\label{prop:Corial}
For a strong divisibility sequence $C$, prime $p$, and integer $n\geq0$, we have
    \begin{equation*} 
    \nu_p(n!_C) = \sum_{k \geq 1} \left\lfloor \frac{n}{\alpha(p^k)} \right\rfloor.
    \end{equation*} 
\end{proposition}

\begin{proof}
Given integers $m$ and $n$, let $d_m(n)$ be the number of indices $1 \leq j \leq n$ such that $m \mid C_j.$
The $ p $-adic valuation of the $C$-orial is then
\[
\nu_p(n!_C) = \sum_{k \geq 1} d_{p^k}(n).
\]

Knuth and Wilf~\cite{KnuthWilf1989} proved that the strong divisibility property of a sequence $C$ is equivalent to the following characterization: for each integer $ m > 0 $, either there exists a positive integer $ \alpha(m) $ such that
\[
m \mid C_j \quad \Longleftrightarrow \quad \alpha(m) \mid j \quad \text{for all } j > 0,
\]
or else $ m \nmid C_j $ for any $ j > 0 $, in which case one formally sets $ \alpha(m) = \infty $. 

The $d$ functions for a strong divisibility sequence then have the simple form $$d_m(n) = \left\lfloor\frac{n}{\alpha(m)}\right\rfloor.$$

Substituting $ d_{p^k}(n) = \left\lfloor \frac{n}{\alpha(p^k)} \right\rfloor $, we arrive at the proposed equation.
\end{proof}

The valuation of the $C$-nomial and $C$-multinomial can be stated in terms the valuation of the $C$-orial as 
\begin{align*}
\nu_p\!\left(\binom{n}{m}_{\!C}\right) &= \nu_p(n!_C)-\nu_p(m!_C)-\nu_p((n-m)!_C)\\
\nu_p\!\left(\binom{n}{m_1,m_2,\dots,m_k}_C \right) &= \nu_p(n!_C)-\nu_p(m_1!_C)-\nu_p(m_2!_C)-\cdots-\nu_p(m_k!_C).
\end{align*}

Specifically, Proposition~\ref{prop:Corial} yields explicit expressions for the valuation of the $ C $-orial depending on the class of primes and subsequently simplifying $\alpha(p^k)$.  

\section{Ideal Primes}\label{section:ideal} We first consider ideal primes and prove the matrix product formulas for the generating functions that count $C$-nomial and $C$-multinomial coefficients by their $p$-adic valuations for ideal prime $p$. From the definition of ideal prime and the constraints on $a_k(p)_{k\geq2}$, the rank of apparition for powers of $p$ satisfies \[
\alpha(p^k) =
\begin{cases}
\alpha(p) & \text{if } 1 \leq k \leq s(p) \\
p^{k - s(p)} \alpha(p) & \text{if } k > s(p).
\end{cases}
\]
This specifies our Legendre-like formula in Proposition~\ref{prop:Corial} into one for ideal primes:

\begin{lemma} 
\label{lemma:idealCorial}
Let $ p $ be an ideal prime for a strong divisibility sequence $ C $. Then for any integer $ n \geq 0 $,
\[
\nu_p(n!_C) = s(p) \left\lfloor \frac{n}{\alpha(p)} \right\rfloor + \nu_p\!\left( \left\lfloor \frac{n}{\alpha(p)} \right\rfloor ! \right).
\]
\end{lemma}

\begin{proof}

We split the sum in Proposition ~\ref{prop:Corial} based on the index $k$ as
\[
\nu_p(n!_C) = \sum_{1 \leq k \leq s(p)} \left\lfloor \frac{n}{\alpha(p)} \right\rfloor + \sum_{k > s(p)} \left\lfloor \frac{n}{p^{k - s(p)} \alpha(p)} \right\rfloor.
\]
Let $ m = \left\lfloor \frac{n}{\alpha(p)} \right\rfloor $. Then by Legendre's formula, the second addend becomes
\[
\sum_{j \geq 1} \left\lfloor \frac{m}{p^j} \right\rfloor = \nu_p(m!),
\]
which completes the proof.
\end{proof}

Using this valuation of the $C$-orial and a modular casework analysis, we now relate the generating function counting $C$-nomials
\[
T_{p,C}(\alpha(p)\cdot n + r, x) = \sum_{m=0}^{\alpha(p)n + r} x^{\nu_p\!\left( \binom{\alpha(p)n + r}{m}_C \right)}
\]
to generating functions counting binomials
$$\sum_{m=0}^{n} x^{\nu_p\!\left(\binom{n}{m}\right)}, \hspace{1cm} \sum_{m=0}^{n-1} x^{1+\nu_p\!\left(n\binom{n-1}{m}\right)},$$ 
which arise naturally as we analyze the $p$-adic valuation of the $C$-nomial. 

\begin{lemma}
\label{lemma:idealCnomialSummationInBinomials}
Let $p$ be an ideal prime for a strong divisibility sequence $C$. For any integer $n\geq0$, let $n=\alpha(p)n'+r$ where $0 \leq r <\alpha(p)$ and $ n' \in \N $. Then
$$\sum_{m=0}^{n} x^{\nu_p\!\left( \binom{n}{m}_C \right)} = (r+1)\sum_{m'=0}^{n'} x^{\nu_p\!\left(\binom{n'}{m'}\right)} + (\alpha(p) - r - 1)x^{s(p) - 1}\sum_{m'=0}^{n'-1} x^{1+\nu_p\!\left(n'\binom{n'-1}{m'}\right)}.$$
\end{lemma}

\begin{proof}
We first analyze the exponent of the left-hand-side $$\nu_p\!\left(\binom{n}{m}_C\right) = \nu_p(n!_C)-\nu_p(m!_C)-\nu_p\!\left((n-m)!_C\right).$$ 
It is convenient to write $ m_1 \colonequal m=\alpha(p) m_1' + r_1 $, and $m_2\colonequal(n-m)=\alpha(p)m_2'+r_2$, where $ 0 \leq r_1,r_2 < \alpha(p) $ and $ m_1',m_2' \in \N $, and consider
\begin{align*}
\nu_p\!\left(\binom{\alpha(p)n'+r}{\alpha(p)m_1'+r_1}_C\right) &= \nu_p\!\left((\alpha(p)n' + r)!_C\right)-\nu_p\!\left((\alpha(p)m_1' + r_1)!_C\right)-\nu_p\!\left((\alpha(p)m_2' + r_2)!_C\right).
\end{align*}
After substituting Lemma~\ref{lemma:idealCorial}, this becomes
\begin{align*}
    & s(p)\left(\left\lfloor\frac{\alpha(p)n'+r}{\alpha(p)}\right\rfloor-\left\lfloor\frac{\alpha(p)m_1'+r_1}{\alpha(p)}\right\rfloor-\left\lfloor\frac{\alpha(p)m_2' + r_2}{\alpha(p)}\right\rfloor\right) \\ & +\nu_p\left(\left\lfloor\frac{\alpha(p)n'+r}{\alpha(p)}\right\rfloor!\right) - \nu_p\left(\left\lfloor\frac{\alpha(p)m_1'+r_1}{\alpha(p)}\right\rfloor!\right) - \nu_p\left(\left\lfloor\frac{\alpha(p)m_2'+r_2}{\alpha(p)}\right\rfloor!\right).
\end{align*}
Simplifying each floor, this valuation equals
\begin{equation} \label{eq:valuationCnomial}
\nu_p\!\left(\binom{\alpha(p)n'+r}{\alpha(p)m_1'+r_1}_C\right)= s(p)\left(n'-m_1'-m_2'\right) + \nu_p\left(n'!\right) - \nu_p\left(m_1'!\right) - \nu_p\left(m_2'!\right).
\end{equation}
Our analysis now splits based the value of $\lambda:= n'-m_1'-m_2'$. 
Because $n=m_1+m_2$, we may state
$$r+\lambda\alpha(p)=r_1+r_2.$$
Therefore, by the bounds on $r$, $r_1$, and $r_2$, $\lambda$ may only take on the value of $0$ or $1$. 

The first case we consider is $\lambda=0$. In this case, $n'-m_1'-m_2'=0$. Thus Equation ~\eqref{eq:valuationCnomial} simplifies to
\begin{equation*} 
\nu_p\!\left(\binom{\alpha(p)n'+r}{\alpha(p)m_1'+r_1}_C\right)=
\nu_p\left(n'!\right) - \nu_p\left(m_1'!\right) - \nu_p\left(m_2'!\right) = \nu_p\left(\frac{n'!}{m_1'!m_2'!}\right) = \nu_p\!\left(\binom{n'!}{m_1'!}\right).
\end{equation*}
We substitute this for the right hand side of the $p$-adic valuation of the $C$-nomial, finding \begin{align*}
    \nu_p\!\left(\binom{\alpha(p)n'+r}{\alpha(p)m_1'+r_1}_C\right) &= \nu_p\!\left(\binom{n'}{m_1'}\right).
\end{align*}

The other case is $\lambda=1$.
In this case, $n'-m_1'-m_2'=1$. Thus the first addend on the right hand side of Equation ~\eqref{eq:valuationCnomial} evaluates to $s(p)$. Since $m_1'+m_2'=n-1$, notice that the latter addends simplify as 
$$\nu_p\left(n'!\right) - \nu_p\left(m_1'!\right) - \nu_p\left(m_2'!\right)= \nu_p\left(\frac{n' (n'-1)!}{m_1'! m_2'!}\right)= \nu_p\!\left(n'\binom {n'-1} {m_1'}\right).$$
Therefore, the $p$-adic valuation of the $C$-nomial in this case is
\begin{align*}
    \nu_p\!\left(\binom{\alpha(p)n'+r}{\alpha(p)m_1'+r_1}_C\right) &= s(p)+\nu_p\!\left(n'\binom{n'-1}{m_1'}\right).
\end{align*}

Now let's interpret what these cases mean and how they allow us to simplify the summation 
$$\sum_{m=0}^{n} x^{\nu_p\!\left( \binom{n}{m}_C \right)} = \sum_{\alpha(p)m_1'+r_1=0}^{\alpha(p)n'+r} x^{\nu_p\!\left(\binom{\alpha(p)n'+r}{\alpha(p)m+r_1}_C\right)}. $$ 
When $\lambda=0$, $r=r_1+r_2$. Therefore, by the bounds on $r_i$, $r_1\in\{0,1,\ldots,r\}$ and from the iterating variable $\alpha(p)m_1' + r_1$ running up to $\alpha(p)n'+r$, the valid values of $m_1'$ are $0 \leq m_1' \leq n'$. Therefore, this evaluates the summation in this case to $$(r+1)\sum_{m_1'=0}^{n'} x^{\nu_p\!\left(\binom{n'}{m_1'}\right)}.$$
When $\lambda=1$, $r+\alpha(p)=r_1+r_2$. In this case, valid values are $r_1\in\{r+1, \ldots,\alpha(p)-1\}$, and from the iterating variable $\alpha(p)m_1' + r_1$ running up to $\alpha(p)n'+r$, the valid values of $m_1'$ are $0 \leq m_1' \leq n'-1$. Therefore, this evaluates the summation in this case to $$(\alpha(p)-r-1)\sum_{m_1'=0}^{n'-1} x^{s(p)+\nu_p\!\left(n'\binom{n'}{m_1'}\right)}.$$
Adding these two components completes the proof.
\end{proof}

The proof of Theorem~\ref{thm:ideal_cnomial}, which states \[
T_{p,C}(\alpha(p) \cdot n + r, x)
= \begin{bmatrix}
    r+1&(\alpha(p) - r - 1)x^{s(p) - 1}
\end{bmatrix} M_p(n_0) M_p(n_1) \cdots M_p(n_\ell) \begin{bmatrix} 1 \\ 0 \end{bmatrix}
\] follows from connecting Lemma~\ref{lemma:idealCnomialSummationInBinomials} to a linear algebra-powered efficiency improvement.

\begin{proof}[Proof of Theorem~\ref{thm:ideal_cnomial}]
For the midst of his proof ~\cite[Theorem 1]{Rowland2018}, Rowland essentially showed \begin{equation} \label{eq:RowlandPartialProduct}    
M_p(n_0) M_p(n_1) \cdots M_p(n_\ell) \begin{bmatrix} 1 \\ 0 \end{bmatrix} = \begin{bmatrix} \sum\limits_{m=0}^{n} x^{\nu_p\!\left(\binom{n}{m}\right)} \\ \sum\limits_{m=0}^{n-1} x^{1+\nu_p\!\left(n\binom{n-1}{m}\right)} \end{bmatrix}.
\end{equation}

In the binomial case as Rowland studied, a leading factor of the vector $\begin{bmatrix}
    1 & 0
\end{bmatrix}$ on the product was sufficient to isolate the requested generating function counting binomials. In our generalized case, this vector changes to encompass the complexity of a general strong divisibility sequence. Surprisingly however, we only need to change this vector; we can build off the same matrices Rowland used to make up the latter partial product. In other words, the matrices are independent of $C$. Specifically, rewriting the right hand side of Lemma~\ref{lemma:idealCnomialSummationInBinomials} with the coefficients as the entries of the vector, we have \[
 \sum_{m=0}^{\alpha(p)n+r} x^{\nu_p\!\left( \binom{\alpha(p)n+r}{m}_C \right)}= \begin{bmatrix}
    r+1&(\alpha(p) - r - 1)x^{s(p) - 1}
\end{bmatrix}\begin{bmatrix} \sum\limits_{m=0}^{n} x^{\nu_p\!\left(\binom{n}{m}\right)} \\ \sum\limits_{m=0}^{n-1} x^{1+\nu_p\!\left(n\binom{n-1}{m}\right)} \end{bmatrix}.
\]
Notice the left-hand-side summation is $T_{p,C}(\alpha(p)\cdot n+r,x)$. Substituting this and the right hand side's latter factor with Equation \eqref{eq:RowlandPartialProduct} completes the matrix product formula for counting $C$-nomial coefficients according to their $p$-adic valuation for an ideal prime $p$ as \[
 T_{p,C}(\alpha(p)\cdot n+r,x) = \begin{bmatrix}
    r+1&(\alpha(p) - r - 1)x^{s(p) - 1}
\end{bmatrix}M_p(n_0) M_p(n_1) \cdots M_p(n_\ell) \begin{bmatrix} 1 \\ 0 \end{bmatrix}.
\]
\end{proof}
We continue by extending the theorem to its $C$-multinomial analogue. Rowland defined a series of generating functions counting binomials supplemental to his multinomial proof. Specifically, for $k\geq2$ and each $\lambda$ in the range $0 \leq \lambda\leq k-1$, $$T_{p,k,\lambda}(n,x)\colonequal\begin{cases}
    0&\text{if }0\leq n\leq \lambda-1
    \\
    x^{\nu_p\!\left(\frac{n!}{(n-\lambda)!}\right)+\lambda}T_{p,k}(n-\lambda,x)&\text{if }n\geq \lambda.
\end{cases}$$
Recall \[
T_{p,k}(n-\lambda, x) = \sum_{\substack{\mathbf{m} \in \mathbb{N}^k \\ \total\mathbf{m} = n-\lambda}} x^{\nu_p\!\left(\binom{n-\lambda}{m_1,\, m_2,\, \dots,\, m_k}\right)}.
\] and note $T_{p,k,0}(n,x)=T_{p,k}(n,x)$. 
We further use the coefficient defined in Section~\ref{section:introduction}. Namely, for each $0\leq \lambda \leq k-1$, let $$u_{p,k,C,\lambda}(r)=x^{(s(p)-1)\lambda}\sum_{j=0}^{\lambda}(-1)^j\binom{k}{j}\binom{r+(\lambda-j)\alpha(p)+k-1}{k-1}.$$

\begin{lemma} 
\label{lemma:idealCmultinomialSummationInBinomials}
Let $p$ be an ideal prime for a strong divisibility sequence $C$ and integer $k\geq2$. For any integer $n\geq0$, let $n=\alpha(p)n'+r$ where  $0 \leq r <\alpha(p)$ and $ n' \in \N $. Then
\[\sum_{\substack{\mathbf{m} \in \mathbb{N}^k \\ \total\mathbf{m} = n}} x^{\nu_p\!\left(\binom{n}{m_1,\, m_2,\, \dots,\, m_k}_C\right)} = \sum\limits_{\lambda=0}^{k-1}u_{p,k,C,\lambda}(r)T_{p,k,\lambda}(n,x).\]
\end{lemma}

\begin{proof}

We similarly notate $m_i\colonequal m_i'\alpha(p)+r_i$ where $ 0 \leq r_i < \alpha(p) $ and $ m_i' \in \N $ for each $1\leq i\leq k$.  Similar to Equation \eqref{eq:valuationCnomial}, the $p$-adic valuation of the $C$-multinomial may be written as $$\nu_p\!\left(\binom{n}{m_1, m_2, \cdots,m_k}_C\right)=s(p)\left(n'-\total \mathbf{m'}\right)+\nu_p(n'!)-\sum\nu_p(m_i'!).$$ 

Analogously to our proof of Lemma~\ref{lemma:idealCnomialSummationInBinomials}, our casework is split based on the value of $\lambda=n'-\total \mathbf{m'}$. Specifically, since $n=\total \mathbf{m}$, we may state
$$r+\lambda\alpha(p)=\total \mathbf{r}.$$ 
Therefore, by the bounds on $r$ and $r_i$, $\lambda$ can take on values in $\{0,1,\ldots,k-1\}$. We analyze each by considering $\lambda$ as a variable. 
Because $$\nu_p(n'!)-\sum\nu_p(m_i'!)=\nu_p\!\left(\frac{n'!}{(n'-\lambda)!}\right) + \nu_p\!\left(\binom{n'-\lambda}{m_1', m_2', \cdots,m_k'}\right),$$ the $p$-adic valuation of the $C$-multinomial is then \begin{equation}\label{eq:idealValuationCmulti}
s(p)\lambda+\nu_p\!\left(\frac{n'!}{(n'-\lambda)!}\right) + \nu_p\!\left(\binom{n'-\lambda}{m_1', m_2', \cdots,m_k'}\right).    
\end{equation}
This allows us to simplify the summation 
\begin{align*}
\sum_{\substack{\mathbf{m}\in\N^k \\ \total\mathbf{m}= n}} & x^{\nu_p\!\left(\binom{n}{m_1,\, m_2,\, \dots,\, m_k}_C\right)} 
\end{align*}
by considering the valuation of a $C$-multinomial in terms of the valuation of multinomials.

For this, we must count the number of tuples $\mathbf{r}\in\{0, 1, \ldots \alpha(p)-1\}^k$ such that $\total \mathbf{r}=r+\lambda\alpha(p)$. This is a type of stars-and-bars counting problem which can be enumerated via the principal of inclusion-exclusion, finding the number of tuples as 
\begin{equation} \label{eq:count_tuples}
\sum_{j=0}^{\lambda}(-1)^j\binom{k}{j}\binom{r+(\lambda-j)\alpha(p)+k-1}{k-1}.
\end{equation}
This count is multiplied atop 
\begin{equation} \label{eq:multinomial_summation}
x^{\nu_p\!\left(\frac{n'!}{(n'-\lambda)!}\right)+ s(p)\lambda}\sum_{\substack{\mathbf{m'}\in\N^k \\ \total\mathbf{m'}= n'-\lambda}} x^{\nu_p(\binom{n'-\lambda}{m_1',\, m_2',\, \dots,\, m_k'})}    
\end{equation}
which comes from Equation \eqref{eq:idealValuationCmulti}, since for every tuple $\mathbf{m'}$ where $\total\mathbf{m'}=n'-\lambda$, each tuple $\mathbf{r}$ where $\total\mathbf{r}=r+\lambda\alpha(p)$ corresponds to a unique tuple $\mathbf{m}$ where $\total\mathbf{m} = n$. 
Doing this for every $\lambda\in\{0, 1,\ldots,k-1\}$ enumerates every tuple covered in the $C$-nomial generating function summation. Therefore, adding the product of \eqref{eq:count_tuples} and \eqref{eq:multinomial_summation} for every $\lambda$ completes the proof.
\end{proof}
This leads easily into the proof of Theorem~\ref{thm:ideal_cmultinomial}, which states \[
T_{p,k,C}(\alpha(p) \cdot n + r, x)
= u_{p,k,C}(r) M_{p,k}(n_0) M_{p,k}(n_1) \cdots M_{p,k}(n_\ell) e^\intercal
\] where $e=\begin{bmatrix}
    1&0&0&\cdots&0
\end{bmatrix}\in\mathbb{Z}^k$, by analyzing a section of Rowland's multinomial proof ~\cite[Theorem 3]{Rowland2018}.

\begin{proof}[Proof of Theorem~\ref{thm:ideal_cmultinomial}]
For his definition of $M_{p, k}(d)$, the product $$M_{p,k}(n_0) M_{p,k}(n_1) \cdots M_{p,k}(n_\ell) e^\intercal,$$ for any $ n \geq 0 $ with the standard base-$p$ representation of $n = n_{\ell}\cdots n_1n_0$, is equivalent to the column vector \[
\begin{bmatrix}
    T_{p,k,0}(n,x)
    \\
    T_{p,k,1}(n,x)
    \\
    \vdots
    \\
    T_{p,k,k}(n,x)
\end{bmatrix}
\]
of multinomial generating functions. Multiplying the row vector $$u_{p,k,C}(r)=\begin{bmatrix}
    u_{p,k,C,0}(r)&u_{p,k,C,1}(r) & \cdots & u_{p,k,C,k-1}(r)
\end{bmatrix}$$
by the column vector above precisely matches Lemma~\ref{lemma:idealCmultinomialSummationInBinomials}, completing the matrix product formula for counting $C$-multinomial coefficients according to their $p$-adic valuation for an ideal prime $p$.
\end{proof}
\section{Acceptable Primes}\label{section:acceptable} For an acceptable prime $p$, the matrix product formula for counting $C$-nomials are not essentially simpler than the formula for counting $C$-multinomials. We thus do not include a separate proof, as plugging in $k=2$ reduces to the binomial analogue. 

The rank of apparitions for powers of $p$ follow that, for $2 \leq k \leq s(p)$, the apparition of $p^k$, $\alpha(p^k)$, may be any value, as long as it is a multiple of the apparition of the previous power $\alpha(p^{k-1})$.  For $k > s(p)$, similar to ideal primes, each rank of apparition is $p$ times the previous, or $\alpha(p^k)=p\alpha(p^{k-1})$. In other words, for an acceptable prime $p$, we have
\begin{align*}
\alpha(p^{k-1})\mid\alpha(p^k) \quad &\text{if } 2 \leq k \leq s(p) \\
\alpha(p^k)=p^{k-s(p)}\alpha(p^{s(p)}) \quad&\text{if } k > s(p).
\end{align*}
The Legendre-like formula for acceptable primes is similar but less specific as we expand into more general behavior for $\alpha(p^k)$. 

\begin{lemma}
\label{lemma:acceptableCorial}
Let $ p $ be an acceptable prime for a strong divisibility sequence $ C $. Then for any integer $ n \geq 0 $,
\[
\nu_p(n!_C) = \sum_{1\leq k\leq s(p)} \left\lfloor \frac{n}{\alpha(p^k)} \right\rfloor + \nu_p\left( \left\lfloor \frac{n}{\alpha(p)} \right\rfloor ! \right).
\]
\end{lemma}

\begin{proof}
We similarly split the summation in Proposition \ref{prop:Corial} based on $k$ as $$\nu_p\left(n!_C\right) = \sum_{1\leq k \leq s(p)} \left\lfloor\frac{n}{\alpha(p^k)}\right\rfloor + \sum_{k > s(p)} \left\lfloor\frac{n}{\alpha(p^k)}\right\rfloor.$$
Given that $p$ is acceptable, we may use $\alpha(p^k)$'s form to substitute in the latter addend
$$\nu_p\left(n!_C\right) = \sum_{1\leq k \leq s(p)} \left\lfloor\frac{n}{\alpha(p^k)}\right\rfloor + \sum_{k > s(p)} \left\lfloor\frac{n}{p^{k-s(p)} \alpha(p^{s(p)})}\right\rfloor.$$
The latter term can be similarly simplified as with ideal primes using Legendre's formula  
$$\sum\limits_{k > s(p)} \left\lfloor\frac{n}{p^{k-s(p)} \alpha(p^{s(p)})}\right\rfloor = \nu_p\left(\left\lfloor\frac{n}{\alpha(p^{s(p)})}\right\rfloor!\right),$$
which completes the proof.
\end{proof}

Notably, we now have not only $$\left\lfloor\frac{n}{\alpha(p)}\right\rfloor$$ but all of $$\left\lfloor\frac{n}{\alpha(p)}\right\rfloor, \left\lfloor\frac{n}{\alpha(p^{2})}\right\rfloor, \ldots, \left\lfloor\frac{n}{\alpha(p^{s(p)})}\right\rfloor$$
present within our $p$-adic valuation of a $C$-orial for an acceptable prime $p$.
This presents a difficulty when simplifying the $p$-adic valuation of a $C$-multinomial. Namely, our casework requires not only comparing each $m_i \bmod \alpha(p)$ vs.\ $n \bmod \alpha(p)$
as before, but also 
$m_i \bmod \alpha(p^2)$ vs.\ $n \bmod \alpha(p^2)$, $\ldots$, $m_i \bmod \alpha(p^{s(p)})$ vs.\ $n \bmod \alpha(p^{s(p)})$.

To address this, we use the initial vector as defined in Section~\ref{section:introduction}. Namely, let
\[
f_{p,k,C,\lambda}(r,\mathbf{r})=\lambda\sum_{j=1}^{s(p)}\frac{\alpha(p^{s(p)})}{\alpha(p^{j})} + \sum_{j=1}^{s(p)}\left( \left\lfloor\frac{r}{\alpha(p^{j})} \right\rfloor - \sum_{i=1}^k\left\lfloor\frac{r_i}{\alpha(p^{j})} \right\rfloor \right).
\]  
Further, for integer $k \geq 2$, we define a coefficient for each $0 \leq \lambda\leq k-1$
$$u_{p,k,C,\lambda}(r)=\sum_{\substack{
\mathbf{r}\in \{0, 1, \ldots,\alpha(p^{s(p)})-1\}^k \\ 
\total\mathbf{r}
= r+i\alpha(p^{s(p)})}}x^{f_{p,k,C,\lambda}(r, \mathbf{r})-\lambda}.$$

\begin{lemma}\label{lemma:acceptableCmultinomialSummationInBinomial}
Let $p$ be an acceptable prime for a strong divisibility sequence $C$ and integer $k\geq2$. For any integer $n\geq0$, let $n=\alpha(p^{s(p)})n'+r$ where $ n' \in \N $ and $0 \leq r <\alpha(p^{s(p)})$. Then
\[\sum_{\substack{\mathbf{m} \in \N^k \\ \total\mathbf{m} = n}} x^{\nu_p\!\left(\binom{n}{m_1,\, m_2,\, \dots,\, m_k}_C\right)} = \sum\limits_{i=0}^{k-1}u_{p,k,C,i}(r)T_{p,k,i}(n,x).\]
\end{lemma}

\begin{proof}
We rely on generalizing the many similarities to the proof of Lemma~\ref{lemma:idealCmultinomialSummationInBinomials}. 
First, we notate $$m_i\colonequal\alpha(p^{s(p)})m_i'+r_i$$ where $ 0 \leq r_i < \alpha(p^{s(p)}) $ and $ m_i' \in \N $ for each $1 \leq i \leq k$. 
We start by updating the $p$-adic valuation of the $C$-multinomial given that $p$ is acceptable. Recall for an ideal prime, this valuation was \eqref{eq:idealValuationCmulti} $$s(p)\lambda+\nu_p\!\left(\frac{n'!}{(n'-\lambda)!}\right) + \nu_p\!\left(\binom{n'-\lambda}{m_1', m_2', \dots,m_k'}\right). $$ 
While the latter addends are still present in the acceptable valuation, the first addend $s(p)\lambda$ is not. 
Instead, that addend is replaced with 
$$f_{p,k,C,i}(r,\mathbf{r})=\lambda\sum_{j=1}^{s(p)}\frac{\alpha(p^{s(p)})}{\alpha(p^{j})} + \sum_{j=1}^{s(p)}\left( \left\lfloor\frac{r}{\alpha(p^{j})} \right\rfloor - \sum_{h=1}^k\left\lfloor\frac{r_h}{\alpha(p^{j})} \right\rfloor \right)
$$ 
based on the summation in Lemma~\ref{lemma:acceptableCorial} evaluated over $$\nu_p\!\left(\binom{n}{m_1,m_2,\dots,m_k}_C\right)=\nu_p(n!_C)-\nu_p(m_1!_C)-\nu_p(m_2!_C)-\cdots-\nu_p(m_k!_C).$$ 

Like Lemma~\ref{lemma:idealCmultinomialSummationInBinomials}, enumerating every valid tuple $\mathbf{r}$ and $\mathbf{m'}$ combines to every tuple $\mathbf{m}$. Therefore, the product of $$\sum_{\substack{\mathbf{r} \in \{0, 1, \ldots,\alpha(p^{s(p)})-1\}^k \\ \total\mathbf{r} = r+i\alpha(p^{s(p)})}}x^{f(r, r_i)}$$ and 
$$x^{\nu_p\!\left(\frac{n'!}{(n'-\lambda)!}\right)}\sum_{\substack{\mathbf{m'} \in \mathbb{N}^k \\ \total\mathbf{m'} = n'-\lambda}} x^{\nu_p\!\left(\binom{n'-\lambda}{m_1',\, m_2',\, \dots,\, m_k'}\right)}$$
for each $0 \leq \lambda \leq k-1$ is equivalent to $$\sum_{\substack{\mathbf{m} \in \mathbb{N}^k \\ \total\mathbf{m} = n}} x^{\nu_p\left(\binom{n}{m_1,\, m_2,\, \dots,\, m_k}_C\right)}$$ since the coefficients match for each exponent on $x$, counting the $p$-adic valuation of $C$-multinomials.  This summation of products completes the proof.
\end{proof}

The proof of Theorem \ref{thm:acceptable_cmultinomial} which states \[
T_{p,k,C}(\alpha(p^{s(p)}) \cdot n + r, x)
= u_{p,k,C}(r) M_{p,k}(n_0) M_{p,k}(n_1) \cdots M_{p,k}(n_\ell) e^\intercal
\] follows shortly. We do not separately prove Theorem \ref{thm:acceptable_cnomial} as it is subcase of this theorem where $k=2$.

\begin{proof}[Proof of Theorem \ref{thm:acceptable_cmultinomial}]
We reuse the column vector of multinomial generating functions from the proof of Theorem \ref{thm:ideal_cmultinomial} and multiply it with the row vector using the coefficients $$u_{p,k,C}(r)=\begin{bmatrix}
    u_{p,k,C,0}(r)&u_{p,k,C,1}(r) & \cdots & u_{p,k,C,k-1}(r)
\end{bmatrix}.$$ This completes the matrix product formula for counting $C$-multinomial coefficients according to their $p$-adic valuation for an acceptable prime $p$. 
\end{proof}
This is the most general statement our current definitions are capable of finding a formula for. It would be interesting to see if there is another way to count the coefficient $u_{p,k,C,i}(r)$ for acceptable primes, perhaps using the principle of inclusion-exclusion repeatedly over each $\alpha(p^j)$ for $1 \leq j \leq s(p)$. The universality of the matrices, namely their independence on the sequence $C$, presents potentially surprising conclusions regarding generalized binomial coefficients in other combinatorial number theory questions.

\section*{Acknowledgements}
I thank Dr.~Eric Rowland for his invaluable guidance in techniques and motivation across the creation of this work. I also thank Dr.~Adam Tagert and Mrs.~Shavon~Donnell from the National Security Agency (NSA) for awarding me the NSA Research Directorate and the American Mathematical Society (AMS) for awarding me the Karl Menger Award at the 2024 International Science and Engineering Fair (ISEF), and the Mathematics Category judges for awarding me the 4th Place Grand Award at the 2025 ISEF. Finally, I thank Dr.~Michael Lake for inspiring me to explore mathematics research in the first place.

\end{document}